\documentclass[12pt,a4paper]{amsart}

\usepackage{amsfonts, amsthm, amssymb, amsmath, stmaryrd}
\usepackage{mathrsfs,array}
\usepackage{pbox}
\usepackage{times,color,enumerate}
\usepackage{url}
\usepackage{float}
\usepackage{pbox, xypic, changepage}
\usepackage[headings]{fullpage}

\usepackage{ytableau}
\usepackage{color}
\usepackage{mathrsfs}
\usepackage{amssymb}
\usepackage{tikz}
\usepackage{tikz-cd}
\usepackage{bm}
\usepackage{enumerate}
\usetikzlibrary{calc}

\newtheorem{theorem}{Theorem}[section]
\newtheorem*{theorem*}{Theorem}
\newtheorem{proposition}[theorem]{Proposition}

\newtheorem{corollary}[theorem]{Corollary}
\newtheorem{definition}[theorem]{Definition}
\newtheorem{example}[theorem]{Example}
\newtheorem{remark}[theorem]{Remark}

\newtheorem{conjecture}{Conjecture}

\newtheorem*{lemma*}{Lemma}
\newtheorem*{remark*}{Remark}
\newtheorem*{example*}{Example}

\DeclareMathAlphabet{\mathpzc}{OT1}{pzc}{m}{it}
\makeatletter
\def\blfootnote{\xdef\@thefnmark{}\@footnotetext}
\makeatother

\title{The SYZ Conjecture via homological mirror symmetry}
\author{Dori Bejleri}

\usepackage[backref]{hyperref}
\hypersetup{
  colorlinks   = true,          
  urlcolor     = purple,          
  linkcolor    = blue,          
  citecolor   = purple             
}

\newcommand{\mb}[1]{\mathbb{#1}}

\newcommand{\mc}[1]{\mathcal{#1}}

\newcommand{\Hom}{\operatorname{Hom}}

\def\ZZ{{\mathbb Z}}

\def\RR{{\mathbb R}}

\def\CC{{\mathbb C}}

\begin{document}
\maketitle

\section{Introduction} These are expanded notes based on a talk given at the Superschool on Derived Categories and $D$-branes held at the University of Alberta in July of 2016. The goal of these notes is to give a motivated introduction to the Strominger-Yau-Zaslow (SYZ) conjecture from the point of view of homological mirror symmetry. 

The SYZ conjecture was proposed in \cite{SYZ} and attempts to give a geometric explanation for the phenomena of mirror symmetry. To date, it is still the best template for constructing mirrors $\check{X}$ to a given Calabi-Yau $n$-fold $X$. We aim to give the reader an idea of why one should believe some form of this conjecture and a feeling for the ideas involved without getting into the formidable details. We assume some background on classical mirror symmetry and homological mirror symmetry as covered for example in the relevant articles in this volume. 

Should the readers appetite be sufficiently whet, she is encouraged to seek out one of the many more detailed surveys such as  \cite{ab1} \cite{dbranes} \cite{auroux} \cite{ballard} \cite{syztransform} \cite{grosstrop} \cite{grosssurv} \cite{gs1} etc. 

\section{From homological mirror symmetry to torus fibrations}

Suppose $X$ and $\check{X}$ are mirror dual K\"ahler Calabi-Yau $n$-folds. Kontsevich's homological mirror symmetry conjecture \cite{kontsevich} posits that there is an equivalence of categories 
$$
\mc{F}uk(X) \cong D^b(\mathrm{Coh}(\check{X}))\footnote{One should work with the dg/$A_\infty$ enhancements of these categories but we ignore that here.}
$$
between the Fukaya category of $X$ and the derived category of $\check{X}$.\footnote{For this text, one may take this duality as the definition of mirror symmetry. It is not obvious that mirrors exist and they are not unique.} This should make precise the physical expectation that ``the $A$-model on $X$ is equivalent to the $B$-model on $\check{X}$.''  The basic idea of the correspondence is summarized by the following table:
\begin{center}
\small
\begin{adjustwidth}{-.53cm}{0cm}
\begin{tabular}{ c|c|c } 

	& $A$-model on $X$ & $B$-model on $\check{X}$ \\
 \hline
 	objects & Lagrangians with flat $U(m)$-connection $(L,\nabla)$ & (complexes) of coherent sheaves  $\mc{F}$ \\
 \hline
 	morphisms & Floer cohomology groups $HF^*(L,M)$ & Ext groups $Ext^*(\mc{F},\mc{G})$ \\
 \hline
 	endomorphism algebra & $HF^*(L,L) = H^*(L)$ & $Ext^*(\mc{F},\mc{F})$
 	
\end{tabular}
\end{adjustwidth}
\end{center}

\.\\

We can now try to understand how this correspondence should work in simple cases. The simplest coherent sheaves on $\check{X}$ are structure sheaves of points $\mc{O}_p$ and indeed $\check{X}$ is the moduli space for such sheaves:
$$
\{\mc{O}_p \ : \ p \in \check{X}\} \cong \check{X}.
$$
Therefore there must be a family of Lagrangians with flat connections $(L_p, \nabla_p)$ parametrized by $p \in \check{X}$ and satisfying
$$
H^*(L_p) \cong Ext^*(\mc{O}_p,\mc{O}_p).
$$

We may compute the right hand side after restricting to an affine neighborhood $U$ of $p$. Since $U$ is smooth at $p$, then $p$ is the zero set of a section $\mc{O}_U \to V \cong \mc{O}_U^{\oplus n}$. Dualizing, we obtain an exact sequence 
$$
\xymatrix{V^* \ar[r]^s & \mc{O}_U \ar[r] & \mc{O}_p \ar[r] & 0}
$$
that we can extend by the \emph{Koszul resolution}
$$
\xymatrix{0 \ar[r] & \bigwedge^n V^* \ar[r]^(.47){s_n} & \bigwedge^{n-1}V^* \ar[r]^(.6){s_{n-1}} & \ldots  \ar[r]^{s_2} & V^* \ar[r]^s & \mc{O}_U \ar[r] & \mc{O}_p \ar[r] & 0}.
$$
Here 
$$
s_k(v_1 \wedge \ldots \wedge v_k) = \sum_{i = 1}^k s(v_i) v_1 \wedge \ldots \wedge \hat{v}_i \wedge \ldots \wedge v_k.
$$
Truncating and applying $\Hom(-, \mc{O}_p)$ gives us 
$$
\xymatrix{0 & \ar[l]  \bigwedge^n V_p & \ar[l]  \bigwedge^{n-1}V_p &\ar[l]  \ldots & \ar[l]  V_p & \ar[l]  k_p & \ar[l]  0}
$$
where $k_p$ is the skyscraper sheaf, $V_p$ is the fiber of $V$, and all the morphisms are $0$ since $s(w)$ vanishes at $p$ for any $w$. It follows that
$$
Ext^*(\mc{O}_p, \mc{O}_p) = \bigoplus_{k = 0}^n \bigwedge^k V_p
$$
where $V_p$ is an $n$-dimensional vector space (in fact isomorphic by the section $s$ to $T_pU$). 

Therefore we are looking for Lagrangians $L_p$ in $X$ with 
$$
H^*(L_p) \cong \bigoplus_{k = 0}^n \bigwedge^k V_p
$$
where $V_p$ is an $n$-dimensional vector space. If we stare at this for a while, we realize this is exactly the cohomology of a topological $n$-torus; $$H^*(L_p) \cong H^*(T^n).$$ This suggests that points $p \in \check{X}$ might correspond to Lagrangian tori in $X$ with flat connections. \\

We are led to consider the geometry of Lagrangian tori in the symplectic manifold $(X, \omega)$. The first thing to note is that under the isomorphism $TX \cong T^*X$ induced by the symplectic form, the normal bundle of a Lagrangian $L$ is identified with its cotangent bundle: 
$$
N_LX \cong T^*L.
$$
In fact, more is true. There is always a tubular neighborhood of $N_\epsilon(L)$ in $X$ isomorphic to a neighborhood of $L$ in $N_LX$, and under this identification we get that $N_\epsilon(L)$ is \emph{symplectomorphic} to a neighborhood of the zero section in $T^*L$ with the usual symplectic form by the Weinstein neighborhood theorem \cite[Corollary 6.2]{weinstein} 

On the other hand, if $L\cong T^n$ is an $n$-torus then $T^*L \cong \RR^n \times T^n$ is the trivial bundle. Therefore we can consider the projection
$$
\mu : \RR^n \times T^n \to \RR^n.
$$
This is a \emph{Lagrangian torus fibration} of $T^*L$ over an affine space. The restriction of $\mu$ to the tubular neighborhood $N_\epsilon(L)$ under the aforementioned identification equips a neighborhood of $L$ in $X$ with the structure of a Lagrangian torus fibration so $X$ is locally fibered by Lagrangian tori.\\

The SYZ conjecture predicts that this is true globally: given a Calabi-Yau manifold $X$ for which we expect mirror symmetry to hold, then $X$ should be equipped with a global Lagrangian torus fibration $\mu : X \to B$ which locally around smooth fibers looks like the fibration $T^*T^n \to \RR^n$ over a flat base. By the previous discussion, $\check{X}$ should be the moduli space of pairs $(L,\nabla)$ where $L$ is a Lagrangian torus fiber of $\mu$ and $\nabla$ is a flat unitary connection on the $L$. However $\mu$ can, and often will, have singular Lagrangian fibers (see Remark 2.1.\ref{rem:sing}) and understanding how these singular fibers affect $\check{X}$ is the greatest source of difficulty in tackling the SYZ conjecture. 

Let $B_0 \subset B$ be the open locus over which $\mu$ has smooth torus fibers and denote the restriction $\mu_0 : X_0 \to B_0$. Then there is an open subset $\check{X}_0 \subset \check{X}$ for which the description as a moduli space of pairs $(L, \nabla)$ of a smooth Lagrangian torus fiber of $\mu_0$ equipped with a flat unitary connection makes sense. We can ask what structure does $\check{X}_0$ gain from the existence of $\mu : X \to B$? 

Viewing $B_0$ as the space of smooth fibers of $\mu$, there is a natural map $\check{\mu}_0 : \check{X}_0 \to B_0$ given by $(L,\nabla) \mapsto L$. Now a flat unitary connection $\nabla$ is equivalent to a homomorphism
$$
\Hom(\pi_1(L), U(m)).
$$
Since $\check{X}_0$ must be $2n$ real dimensional and $\check{\mu}_0$ is a fibration over an $n$ real dimensional base, the fibers must be $n$ real dimensional and so $m = 1$. That is, the fibers of $\check{\mu}_0$ are given by
$$
\Hom(\pi_1(L), U(1)) \cong L^*
$$
the dual torus of $L$. Ignoring singular Lagrangians, $\check{X}_0 \subset \check{X}$ is equipped with a dual Lagrangian torus fibration $\check{\mu}_0 : \check{X}_0 \to B_0 \subset B$! 

\begin{conjecture} (Strominger-Yau-Zaslow \cite{SYZ}) Mirror Calabi-Yau manifolds are equipped with special Lagrangian fibrations
$$
\xymatrix{ X \ar[rd]_\mu & & \check{X} \ar[ld]^{\check{\mu}} \\ & B &}
$$
such that $\mu$ and $\check{\mu}$ are dual torus fibrations over a dense open locus $B_0 \subset B$ of the base. 

\end{conjecture}

\begin{remark} \begin{enumerate}[(i)]

\item We will discuss the notion of a \emph{special Lagrangian} and the reason for this condition in \ref{sec:slag}.

\item \label{rem:sing} Note that unless $\chi(X) = 0$, then the fibration $\mu$ must have singularities. Indeed the only compact CY manifolds with smooth Lagrangian torus fibrations are tori. 

\item From the point of view of symplectic geometry, Lagrangian torus fibrations are natural to consider. Indeed a theorem of Arnol'd and Liouville states that the smooth fibers of \emph{any} Lagrangian fibration of a symplectic manifold are tori \cite[Section 49]{arnold}. 

\end{enumerate} \end{remark}

This conjecture suggests a recipe for constructing mirror duals to a given Calabi-Yau $X$. Indeed we pick a $\mu : X \to B$ and look at the restriction $\mu_0 : X_0 \to B_0$ to the smooth locus. Then $\mu_0$ is a Lagrangian torus fibration which we may dualize to obtain $\check{\mu}_0 : \check{X}_0 \to B_0$. Then we compactify $X_0$ by adding back the boundary $X \setminus X_0$ and hope that this suggests a way to compactify the dual fibration to obtain a mirror $\check{X}$. 

It turns out the story is not so simple and understanding how to compactify $\check{X}_0$ and endow it with a complex structure leads to many difficulties arising from instanton corrections and convergence issues for Floer differentials. Furthermore this strategy to construct the dual depends not only on $X$ but also on the chosen fibration $\mu$ and indeed we can obtain different mirrors by picking different fibrations, or even from the same fibration by picking a different ``compactification'' recipe. This leads to mirrors that are \emph{Landau-Ginzburg models} and allows us to extend the statement of mirror symmetry outside of the Calabi-Yau case (\cite{kapustinli}, \cite{auroux1}, etc). Finally, there are major issues in constructing Lagrangian torus fibrations in general. Indeed it is not known if they exist for a general Calabi-Yau, and in fact they are only expected to exist in the \emph{large complex structure limit} (LCSL) \cite{grosswilson} \cite{kontsevichsoibelman}. This leads to studying SYZ mirror symmetry in the context LCSL degenerations of CY manifolds rather than for a single CY manifold as in the \emph{Gross-Siebert program} \cite{gs1} \cite{gs2}. We discuss these ideas in more detail in Section \ref{sec:construction}

\subsection{Some remarks on special Lagrangians}\label{sec:slag}

As stated, the SYZ conjecture is about \emph{special} Lagrangian (sLag) torus fibrations rather than arbitrary torus fibrations. Recall that a Calabi-Yau manifold has a nonvanishing holomorphic volume form $\Omega \in H^0(X, \Omega^n_X)$. 

\begin{definition} A Lagrangian $L \subset X$ is special if there exists a choice of $\Omega$ such that $$\mathrm{Im}(\Omega)|_L = 0.$$ \end{definition}

There are several reasons to consider special Lagrangians:

\begin{itemize} 
\item SLags minimize the volume within their homology class. In physics this corresponds to the fact that these are \emph{BPS branes} (see Section \ref{sec:tduality}). Mathematically, this corresponds to the existence of a conjectural Bridgeland-Douglas stability condition on the Fukaya category whose stable objects are the special Lagrangians (see for example \cite{joyce}). 

\item SLags give canonical representatives within a Hamiltonian isotopy class of Lagrangians. Indeed a theorem of Thomas and Yau \cite[Theorem 4.3]{thomasyau} states that under some assumptions, there exists at most one sLag within each Hamiltonian deformation class. 

\item The deformation theory of sLag tori is well understood and endows the base $B$ of a sLag fibration with the structures needed to realize mirror symmetry, at least away from the singularities. We will discuss this in more detail in Section \ref{sec:semiflat}.

\end{itemize}

However, it is much easier to construct torus fibrations than it is to construct sLag torus fibrations and in fact its an open problem whether the latter exist for a general Calabi-Yau. Therefore for many partial results and in many examples, one must get by with ignoring the special condition and considering only Lagrangian torus fibration. 

\subsection{A remark on $D$-branes and $T$-duality}\label{sec:tduality}

Strominger-Yau-Zaslow's original motivation in \cite{SYZ} differed slightly form the story above. Their argument used the physics of $D$-branes, that is, boundary conditions for open strings in the $A$- or $B$-model.\footnote{For background on $D$-branes see for example \cite{dbranes} or the other entries in this volume.}

They gave roughly the following argument for Calabi-Yau threefolds . The moduli space of $D0$\footnote{$D0$, $D3$, \ldots denote $0$-dimensional, $3$-dimensional, \ldots $D$-branes.} $B$-branes on $\check{X}$ must the moduli space of some BPS $A$-brane on $X$. The BPS condition and supersymmetry necessitate that this $D3$ brane consists of a special Lagrangian $L$ equipped with a flat $U(1)$ connection. Topological considerations force $b_1(L) = 3$ and so the space of flat $U(1)$ connections
$$
\Hom(\pi_1(L),U(1)) \cong T^3
$$
is a $3$ torus. Thus $\check{X}$ must fibered by $D3$ $A$-branes homeomorphic to tori and by running the same argument with the roles of $X$ and $\check{X}$ reversed, we must get a fibration by dual tori on $X$ as well. 

The connection with homological mirror symmetry, which was discovered later, comes from the interpretation of the Fukaya category and the derived category as the category of topological $D$-branes for the $A$- and $B$-model respectively. The morphisms in the categories correspond to massless open string states between two $D$-branes. 

Now one can consider what happens if we take a $D6$ $B$-brane given by a line bundle $\mc{L}$ on $\check{X}$. By using an argument similar to the one above, or computing 
$$
Ext^*(\mc{L},\mc{O}_p) \cong k[0],\footnote{That is, $k$ in degree  zero and $0$ in other degrees.}
$$
we see that there is a one dimensional space of string states between $\mc{L}$ and $\mc{O}_p$. Therefore the Lagrangian $S$ in $X$ dual to $\mc{L}$ must satisfy
$$
HF^*(S,L) = k[0]
$$
for any fiber $L$ of the SYZ fibration. Remembering that the Floer homology groups count intersection points of Lagrangians, this suggests that $S$ must be a section of the fibration $\mu$. 

In summary, the SYZ Conjecture states that mirror symmetry interchanges $D0$ $B$-branes on $\check{X}$ with $D3$ Lagrangian torus $A$-branes on $X$ and $D6$ $B$-branes on $\check{X}$ with $D3$ Lagrangian sections on $X$. On a smooth torus fiber of the fibration, this is interchanging $D0$ and $D3$ branes on dual $3$-tori. This duality on each torus is precisely what physicists call $T$-duality and one of the major insights of \cite{SYZ} is that in the presence of dual sLag fibrations, mirror symmetry is equivalent to fiberwise $T$-duality. 

\section{Hodge symmetries from SYZ}

The first computational evidence that led to mirror symmetry was the interchange of Hodge numbers 
\begin{equation}\label{eqn:hodge}
\begin{split}
h^{1,1}(X) = h^{1,2}(\check{X}) \\ 
h^{1,2}(X) = h^{1,1}(\check{X})
\end{split}
\end{equation}
for compact simply connected mirror Calabi-Yau threefolds $X$ and $\check{X}$. 
Thus the first test of the SYZ conjecture is if it implies the interchange of Hodge numbers. We will show this under a simplifying assumption on the SYZ fibrations. 

Let $f : X \to B$ be a proper fibration and let $i:B_0 \subset B$ be the locus over which $f$ is smooth so that $f_0 : X_0 \to B_0$ is the restriction. The $p^{th}$ higher direct image of the constant sheaf $R^pf_*\RR$ is a constructible sheaf with
$$
i^*R^pf_*\RR \cong R^p(f_0)_*\RR
$$
for each $p \geq 0$. Furthermore, $R^p(f_0)_*\RR$ is the local system on $B_0$ with fibers the cohomology groups $H^p(X_b,\RR)$ for $b \in B_0$ since $f_0$ is a submersion.

\begin{definition} We say that $f$ is simple if we can recover the constructible sheaf $R^pf_*\RR$ by the formula
$$
i_*R^p(f_0)_*\RR \cong R^pf_*\RR
$$
for all $p \geq 0$. 
\end{definition}

\begin{proposition}\label{prop:hodge} Suppose $X$ and $\check{X}$ are compact simply connected Calabi-Yau threefolds with dual sLag fibrations
$$
\xymatrix{X \ar[rd]_\mu & & \check{X} \ar[ld]^{\check{\mu}} \\ & B &}
$$
such that $\mu$ and $\check{\mu}$ are simple. Assume further that $\mu$ and $\check{\mu}$ admit sections. Then the Hodge numbers of $X$ and $\check{X}$ are interchanged as in (\ref{eqn:hodge}). 

\end{proposition}

Before the proof, we will review some facts about tori. If $T$ is an $n$-torus, there is a canonical identification
$$
T \cong H_1(T,\RR)/\Lambda_T
$$
where $\Lambda_T$ denotes the lattice $H_1(T,\ZZ)/tors \subset H_1(T,\RR)$. Then the isomorphism $$H^1(T,\RR) \cong H_1(T,\RR)^*$$ induces an identification
$$
T^* \cong H^1(T,\RR)/\Lambda_T^*
$$
where $\Lambda_T^* = H^1(T,\ZZ)/tors \subset H^1(T,\RR)$. It follows that $H_1(T^*, \RR) = H^1(T,\RR)$ and $\Lambda_{T^*} = \Lambda_T^*$. More generally, denoting $V = H_1(T,\RR)$, there are isomorphisms 
\begin{align*}
H^p(T,\RR) &\cong \bigwedge^p V^* \\
H^p(T^*,\RR) &\cong \bigwedge^p V.
\end{align*}
After fixing an identification $\bigwedge^n V \cong \RR$, Poincar\'e duality gives rise to isomorphisms
$$
H^p(T,\RR) \cong H^{n-p}(T^*,\RR)
$$
compatible with the identification $\Lambda_T^* = \Lambda_{T^*}$. 

\begin{proof}[Proof of Proposition \ref{prop:hodge}] Applying the above discussion fiber by fiber to the smooth torus bundle $\mu_0 : X_0 \to B_0$, we obtain that an isomorphism of torus bundles
$$
R^1(\mu_0)_* (\RR/\ZZ) := (R^1(\mu_0)_*\RR)/(R^1(\mu_0)_* \ZZ/tors) \cong \check{X}_0
$$
over $B$. Similarly $X_0 \cong R^1(\check{\mu}_0)_*(\RR/\ZZ)$ and Poincar\'e duality gives rise to
$$
R^p(\mu_0)_*\RR \cong R^{3 - p}(\check{\mu}_0)_*\RR.
$$
By the simple assumption on $\mu$ and $\check{\mu}$ it follows that
\begin{equation}\label{eqn:dual}
R^p \mu_*\RR \cong R^{3 - p}\check{\mu}_*\RR.
\end{equation}
We want to use this isomorphism combined with the Leray spectral sequence to conclude the relation on Hodge numbers. 

First note that $H^1(X,\RR) = 0$ by the simply connected assumption and so $H^5(X,\RR) = 0$ by Poincar\'e duality. This implies the Hodge numbers $h^{0,1}(X), h^{1,0}(X), h^{2,3}(X)$ and $h^{3,2}(X)$ are all zero. By Serre duality, $h^{2,0} = h^{0,2}(X) = h^{0,1}(X) = 0$. Furthermore, $h^{1,3} = h^{3,1} = h^1(X,\Omega_X^3) = h^{0,1} = 0$ by the Calabi-Yau condition. Finally, $h^{3,3} = h^{0,0} = 1$ is evident and $h^{0,3} = h^{3,0} = h^0(X, \Omega_X^3) = 1$ again by the Calabi-Yau condition. Putting this together gives us the following relation between Hodge numbers and Betti numbers:
\begin{align*}
h^{1,1}(X) = b_2(X) = b_4(X) &= h^{2,2}(X) \\ 
b_3(X) = 2 + h^{1,2}(X) + h^{2,1}(X) &= 2(1 + h^{1,2}(X))
\end{align*}
Of course the same is also true for $\check{X}$. Thus it would suffice to show 
\begin{equation}\label{b3}
b_3(\check{X}) = 2 + h^{1,1}(X) + h^{2,2}(X) = 2(1 + h^{1,1}(X))
\end{equation}
from which it follows that $h^{1,1}(X) = h^{1,2}(\check{X})$ as well as $h^{1,1}(\check{X}) = h^{1,2}(X)$ by applying the same argument to $X$. 

The sheaves $R^3\mu_*\RR$ and $R^0\mu_*\RR$ are both isomorphic to the constant sheaf $\RR$. As $X$ is simply connected, so is $B$ so we deduce $H^1(B, \RR) = 0$ and similarly $H^2(B,\RR) = 0$ by Poincar\'e duality. Thus $H^1(B, R^0\mu_*\RR) = H^2(B, R^0\mu_*\RR) = H^1(B, R^3\mu_*\RR) = H^2(B, R^3\mu_*\RR) = 0$ and $H^i(B, R^j\mu_*\RR) = \RR$ for $i,j = 0,3$. Next the vanishing $H^1(X, \RR) = H^5(X,\RR)$ imply that $H^0(B, R^1\mu_*\RR) = H^3(B,R^2\mu_*\RR) = 0$. Applying the same reasoning to $\check{\mu}$ and using the isomorphism (\ref{eqn:dual}), we get
\begin{align*}
H^0(B,R^2\mu_*\RR) &= H^0(B, R^1\check{\mu}_*\RR) = 0\\
H^3(B, R^1\mu_*\RR) &= H^3(B, R^2\check{\mu}_*\RR) = 0. 
\end{align*}

Putting this all together, the $E_2$ page of the Leray spectral sequence for $\mu$ becomes
$$
\xymatrix{\RR \ar[drr]^{d_1} & 0 & 0 & \RR \\ 0 & H^1(B, \RR^2\mu_*\RR) & H^2(B, \RR^2\mu_*\RR) & 0 \\ 0 & H^1(B, \RR^1\mu_*\RR) \ar[drr]^{d_2}& H^2(B, \RR^1\mu_*\RR) & 0 \\ \RR & 0 & 0 & \RR}
$$
with the only possibly nonzero differentials depicted above. We claim in fact that $d_1$ and $d_2$ must also be zero. 

Indeed let $S \subset X$ be a section of $\mu$. Then $S$ induces a nonzero section $s \in \RR \cong H^0(B,R^3\mu_*\RR)$ since it intersects each fiber in codimension $3$. Furthermore $S$ must represent a nonzero cohomology class on $X$ and so $s \in \ker(d_1)$. This forces $d_1$ to be the zero map since $H^0(B,R^3\mu_*\RR)$ is one dimensional. Similarly, the fibers of $\mu$ give rise to a nonzero class in $f \in H^3(B, R^0\mu_*\RR) \cong \RR$. Since the class of a fiber is also nonzero in the cohomology of $X$ as the fibers intersect the section nontrivially, then $f$ must remain nonzero in $\mathrm{coker}(d_2)$; that is, $d_2$ must be zero. 

This means the Leray spectral sequence for $\mu$ degenerates at the $E_2$ page and similarly for $\check{\mu}$. In particular, we can compute 
\begin{align*}
h^{1,1}(X) = b_2(X) &= h^1(B, R^1\mu_*\RR) = h^1(B,R^2\check{\mu}_*\RR) \\
h^{2,2}(X) = b_4(X) &= h^2(B,R^2\mu_*\RR) = h^2(B, R^1\check{\mu}_*\RR)
\end{align*}
where we have again used (\ref{eqn:dual}). Therefore we can verify
$$
b_3(\check{X}) = 2 + h^1(B, R^2\check{\mu}_*\RR) + h^2(B, R^1\check{\mu}_*\RR) = 2 + h^{1,1}(X) + h^{2,2}(X)
$$
as required. 

\end{proof}

\begin{remark} The argument above (originally appearing in \cite{grossslag1}) was generalized by Gross in \cite{grossslag2} to obtain a relation between the integral cohomologies of $X$ and $\check{X}$.  \end{remark}

The reader may object that there are several assumptions required in the above result. The existence of a section isn't a serious assumption. Indeed all that was required in the proof is the existence of a cohomology class that behaves like a section with respect to cup products. As we already saw in \ref{sec:tduality}, mirror symmetry necessitate the existence of such Lagrangians on $X$ dual to line bundles on $\check{X}$ and vice versa. The simplicity assumption, on the other hand, is serious and isn't always satisfied. However, this still gives us a good heuristic check of SYZ mirror symmetry. 

\section{Semi-flat mirror symmetry}\label{semiflat}

In this section we will consider the case where $\mu$ and $\check{\mu}$ are smooth sLag fibrations so that $B_0 = B$. This is often called the \emph{semi-flat} case.

In this case we will see that the existence of dual sLag fibrations endows $B$ with the extra structure of an integral affine manifold which results in a toy model of mirror symmetry on $B$. In fact, we will see that the dual SYZ fibrations can be recovered from this integral affine structures. Finally, we will discuss an approach to realize HMS conjecture in the semi-flat case. 

\subsection{The moduli space of special Lagrangians}\label{sec:semiflat}

The starting point is the following theorem of McLean:

\begin{theorem}\label{thm:mclean}(McLean \cite[Section 3]{mclean}) Let $(X, J, \omega, \Omega)$ be a K\"ahler Calabi-Yau $n$-fold. Then the moduli space $\mathcal{M}$ of special Lagrangian submanifolds is a smooth manifold. Furthermore, there are natural identifications
$$
H^{n-1}(L,\RR) \cong T_L\mathcal{M} \cong H^1(L,\RR) 
$$
of the tangent space to any sLag submanifold $L \subset X$. 
\end{theorem}

The idea is that a deformation of $L$ is given by a normal vector field $v \in C^\infty(N_LX,\RR)$. Then we obtain a $1$-form $\alpha \in \Omega^1(L,\RR)$ and an $n-1$-form $\beta \in \Omega^{n-1}(L,\RR)$ by contraction with $\omega$ and $\mathrm{Im}\Omega$ respectively:
\begin{align*}
\alpha &= -i_v \omega \\
\beta &= i_v \mathrm{Im}\Omega. 
\end{align*}
It turns out that $\alpha$ and $\beta$ determine each other and that $v$ induces a sLag deformation of $L$ if and only if $\alpha$ and $\beta$ are both closed. This gives the above isomorphisms by the maps $v \mapsto [\alpha] \in H^1(L,\RR)$ and $v \mapsto [\beta] \in H^{n-1}(L,\RR)$ respectively. \\

Note in particular that the isomorphism $T_L\mathcal{M} \cong H^1(L,\RR)$ depends on the symplectic structure $\omega$ and the isomorphism $T_L\mathcal{M} \cong H^{n-1}(L,\RR)$ depends on the complex structure through the holomorphic volume form $\Omega$. 

\begin{definition}\label{defn:intaffine} An integral affine manifold $M$ is a smooth manifold equipped with transition functions in the affine group $\RR^n \rtimes \mathrm{GL}_n(\ZZ)$. Equivalently it is a manifold $M$ equipped with a local system of integral lattices $\Lambda \subset TM$. 
\end{definition}

The equivalence in definition \ref{defn:intaffine} can be seen by noting that if the transition functions of $M$ are affine transformations, they preserve the integral lattice defined in local coordinates by
\begin{equation}\label{eqn:coord}
\Lambda := \mathrm{Span}_\ZZ\left( \frac{\partial}{\partial y_1}, \ldots, \frac{\partial}{\partial y_n} \right) \subset TU. 
\end{equation}
On the other hand, if there exists a local system of integral lattice $\Lambda \subset TM$ with a compatible flat connection $\nabla$ on $TM$, then on a small enough coordinate patch we can choose coordinates such that $\Lambda$ is the coordinate lattice and the transition functions must be linear isomorphisms on this lattice.  \\

The vector spaces $H^1(L,\RR)$ and $H^{n-1}(L,\RR)$ glue together to form vector bundles on $\mathcal{M}$. Explicitly, if $\mathcal{L} \subset X \times \mathcal{M}$ is the universal family of sLags over $\mathcal{M}$ with projection $\pi : \mathcal{L} \to \mathcal{M}$ then these bundles are $R^1\pi_*\RR$ and $R^{n-1}\pi_*\RR$ respectively. Similarly, the integral cohomology groups $H^1(L,\ZZ)/tors \subset H^1(L,\RR)$ and $H^{n-1}(L,\ZZ)/tors \subset H^{n-1}(L,\RR)$ glue together into local systems of integral lattices $R^1\pi_*\ZZ/tors \subset R^1\pi_*\RR$ and $R^{n-1}\pi_*\ZZ/tors \subset R^{n-1}\pi_*\RR$. Applying Theorem \ref{thm:mclean} fiber by fiber yields two integral affine structures on $\mathcal{M}$:

\begin{corollary}\label{cor:intlattice} There are isomorphisms $R^1\pi_*\RR \cong T\mathcal{M} \cong R^{n-1}\pi_*\RR$ which endow $\mathcal{M}$ with two integral affine structures given by the integral lattices 
\begin{align*}
R^1\pi_*\ZZ/tors &\subset R^1\pi_*\RR \cong T\mathcal{M} \\
R^{n-1}\pi_*\ZZ/tors &\subset R^{n-1}\pi_*\RR \cong T\mathcal{M}.
\end{align*}
Poincare duality induces an isomorphism $T\mathcal{M} \cong T^*\mathcal{M}$ exchanging the lattices and their duals. \end{corollary}

\subsection{Mirror symmetry for integral affine structures}

\subsubsection{From SYZ fibrations to integral affine structures}\label{sec:intaffine} Now let us return to the case of dual SYZ fibrations
$$
\xymatrix{X \ar[rd]_\mu & & \check{X} \ar[ld]^{\check{\mu}} \\ & B & }
$$
where both $\mu$ and $\check{\mu}$ are smooth. Then $\dim B = n = \dim H^1(L,\RR)$ is the dimension of the moduli space of sLag $n$-tori in $X$ and so $B$ must be an open subset of the moduli space $\mathcal{M}$. 

In particular, by Corollary \ref{cor:intlattice}, the symplectic form $\omega$ and the holomorphic volume form $\Omega$ on $X$ induces two integral affine structures on $B$ explicitly given by
\begin{align*}
\Lambda_\omega := R^1\mu_*\ZZ/tors &\subset R^1\mu_*\RR \cong TB \\
\Lambda_\Omega := R^{n-1}\mu_*\ZZ/tors &\subset R^{n-1}\mu_*\RR \cong TB
\end{align*}
We call these the K\"ahler and complex integral affine structures respectively. Similarly the symplectic and holomorphic forms $\check{\omega}$ and $\check{\Omega}$ on $\check{X}$ induce two other integral affine structures
\begin{align*}
\Lambda_{\check{\omega}} := R^1\check{\mu}_*\ZZ/tors &\subset R^1\check{\mu}_*\RR \cong TB \\
\Lambda_{\check{\Omega}} := R^{n-1}\check{\mu}_*\ZZ/tors &\subset R^{n-1}\check{\mu}_*\RR \cong TB
\end{align*}
on $B$. The fact that these torus fibrations are dual implies natural isomorphisms 
\begin{align*}
R^1\mu_*\RR &\cong R^{n-1}\check{\mu}_*\RR \\
R^{n-1}\mu_*\RR &\cong R^1\check{\mu}_*\RR 
\end{align*}
The top isomorphism exchanges $\Lambda_\omega$ and $\Lambda_{\check{\Omega}}$ while the bottom isomorphism exchanges $\Lambda_{\check{\omega}}$ and $\Lambda_\Omega$. We can summarize this as follows: \emph{SYZ mirror symmetry for smooth sLag torus fibrations interchanges the complex and K\"ahler integral affine structures on the base $B$.} \\

\subsubsection{From integral affine structures to SYZ fibrations} We can go in the other direction and recover the mirror SYZ fibrations $\mu$ and $\check{\mu}$ from the integral affine structures on the base $B$. The key is the following proposition:

\begin{proposition}\label{prop:semiflat} Let $(B,\Lambda \subset TB)$ be an integral affine manifold. Then the torus fibration $TB/\Lambda \to B$ has a natural complex structure and the dual torus fibration $T^*B/\Lambda^* \to B$ has a natural symplectic structure. \end{proposition} 

\begin{proof} Locally we can find a coordinate chart $U \subset B$ with coordinates $y_1, \ldots, y_n$ such that $\Lambda$ is a coordinate lattice as in (\ref{eqn:coord}).  Then the coordinate functions on $TU$ are given by $y_1, \ldots, y_n$ and $x_1 = dy_1, \ldots, x_n = dy_n$ and we can define holomorphic coordinates on $TU$ by $z_j = x_j + \sqrt{-1} y_j$. Since the transition functions on $B$  preserve the lattice, they induce transition functions on $TB$ that are holomorphic with respect to these coordinates giving $TB$ the structure of a complex manifold. 

Consider the holomorphic functions defined locally by
$$
q_j := e^{2\pi \sqrt{-1}z_j}.
$$
These functions are invariant under integral affine transition functions as well as global translations by $\Lambda$ and so they give a compatible system of holomorphic coordinates for $TB/\Lambda$.

Similarly, in local coordinates $U$ where $\Lambda$ is the coordinate lattice, then $\Lambda^* \subset T^*U$ is generated by $dy_1, \ldots, dy_n$ as a lattice in $T^*U$. Therefore the standard symplectic structure on $T^*B$ is invariant by $\Lambda^*$ and descends to $T^*B/\Lambda^*$. 

\end{proof}

Now suppose $B$ is a smooth manifold equipped with two integral affine structures $\Lambda_0, \Lambda_1 \subset TB$ as well as an isomorphism $TB \cong T^*B$ such that $\Lambda_0 \cong (\Lambda_1)^*$ and $\Lambda_1 \cong (\Lambda_0)^*$. Then we have dual torus fibrations 
$$
\xymatrix{ X \ar[dr]_\mu & & \check{X}  \ar[ld]^{\check{\mu}} \\ & B & }
$$
where $X := TB/\Lambda_0 \cong T^*B/(\Lambda_1)^*$ and $\check{X}  := T^*B/(\Lambda_0)^* \cong TB/(\Lambda_1)$. This construction satisfies the following properties:

\begin{enumerate}[(a)]
\item if $\Lambda_0$ and $\Lambda_1$ are the integral affine structures associated to SYZ dual torus fibrations as in Section \ref{sec:intaffine}, then this construction recovers the original fibrations; 

\item $\Lambda_0$ determines the complex structure of $X$ and the symplectic structure of $\check{X}$;

\item $\Lambda_1$ determines the symplectic structure of $X$ and the complex structure of $\check{X}$. 

\end{enumerate}
As a result we recover one of the main predictions of mirror symmetry: \emph{deformations of the complex structure on $X$ are the same as deformations of the symplectic structure on $\check{X}$ and vice versa.}

\begin{remark} There is an extra piece of structure on $B$ that we haven't discussed. This is a Hessian metric $g$ realizing the identification $TB \cong T^*B$. Recall that a Hessian metric is a Riemannian metric that is locally the Hessian of some smooth potential function $K$. The two integral affine structures on $B$ endow it with two different sets of local coordinates and the potential functions in these coordinates are related by the Legendre transform. In fact the complex and symplectic structures constructed in Proposition \ref{prop:semiflat} can be recovered from the potential function so mirror symmetry in this context is governed by the Legendre transform \cite{hitchin} \cite[Section 6.1.2]{dbranes}.  \end{remark} 

\subsection{The SYZ transform} To finish our discussion of semi-flat mirror symmetry, we turn our attention to the homological mirror symmetry conjecture. The goal is to construct a geometric functor 
$$
\Phi : \mc{F}uk(X) \to D^b(\mathrm{Coh}(\check{X}))
$$
from the Fukaya category of $X$ to the derived category of coherent sheaves on $\check{X}$ using the geometry of the dual fibrations. The first step is to produce an object of $D^b(\mathrm{Coh}(\check{X}))$ from a Lagrangian $L \subset X$ equipped with a flat unitary connection. We will attempt to do this by exploiting the fact that a point $p \in X$ corresponds to a flat $U(1)$-connection on the dual fiber.

Let $L \subset X$ be a Lagrangian section of $\mu$ corresponding to a map $\sigma : B \to X$, equipped with the trivial connection. By restricting $L$ to each fiber of $\mu$, we obtain a family of flat $U(1)$-connections 
$$
\{\nabla_{\sigma(b)}\}_{b \in B}
$$
on the fibers of $\check{\mu} : \check{X} \to B$. These glue together to give a flat $U(1)$-connection on a complex line bundle $\mathcal{L}$ on $\check{X}$. It turns out this connection gives $\mathcal{L}$ the structure of a holomorphic line bundle on $\check{X}$ (endowed with the complex structure constructed in the last subsection).

This construction was generalized by \cite{ap} (see also \cite{slaghym}) as follows. As $X$ is the moduli space of flat $U(1)$-connections on the fibers of $\check{\mu} : \check{X} \to B$, there exists a universal bundle with connection $(\mathcal{P}, \nabla^{\mathcal{P}})$ on $X \times_B \check{X}$. Given $(L,\mathcal{E},\nabla)$ where $L \subset X$ is a multisection transverse to the fibers of $\mu$ and $(\mathcal{E},\nabla)$ is a flat unitary vector bundle on $L$, define the \emph{SYZ transform} by
$$
\Phi^{SYZ}(L, \mathcal{E}, \nabla) := (pr_{\check{X}})_*((pr_L)^*\mathcal{E} \otimes (i \times id)^*\mathcal{P})
$$
where $pr_L, pr_{\check{X}} : L \times_B \check{X} \to L, \check{X}$ are the projections and $(i \times id) : L \times_B \check{X} \to X \times_B \check{X}$ is the inclusion. Note that $\Phi^{SYZ}(L, \mathcal{E},\nabla)$ comes equipped with a connection we denote $\nabla_{(L,\mathcal{E},\nabla)}$. 

\begin{theorem}\label{thm:transform}(\cite[Theorem 1.1]{ap}) If $L \subset X$ is Lagrangian, then $\nabla_{(L,\mathcal{E},\nabla)}$ endows \linebreak $\Phi^{SYZ}(L,\mathcal{E},\nabla)$ with the structure of a holomorphic vector bundle on $\check{X}$. When $X$ and $\check{X}$ are dual elliptic curves fibered over $S^1$, then every holomorphic vector bundle on $\check{X}$ is obtained this way.  
\end{theorem}

Viewing holomorphic vector bundles as objects in $D^b(\mathrm{Coh}(\check{X}))$, we can hope to extend the SYZ transform to an equivalence $\Phi : \mathcal{F}uk(X) \to D^b(\mathrm{Coh}(\check{X}))$, thus realizing the HMS conjecture. While this hope hasn't been realized in general, it has in some special cases. 

When $X$ and $\check{X}$ are dual elliptic curves fibered over $S^1$, a HMS equivalence $\Phi$ is constructed by hand in \cite{elliptic}. One can check that their functor $\Phi$ does indeed extend the SYZ transform $\Phi^{SYZ}$. In fact, assuming Theorem \ref{thm:transform}, it is not so hard to construct $\Phi$ at least on the level of objects. Each coherent sheaf on the curve $X$ can be decomposed as a direct sum of a torsion sheaf and a vector bundle. Vector bundles are taken care of by Theorem \ref{thm:transform}. Torsion sheaves are successive extensions of skyscrapers at points which correspond to $S^1$ fibers of $\mu : X \to B$. For more recent work on understanding the SYZ transform see \cite{syztransform} and the references therein.

\section{Constructing mirrors}\label{sec:construction}

We now move on to the general problem of constructing mirrors. Given a K\"ahler Calabi Yau $n$-fold $(X, J, \omega, \Omega)$, the SYZ conjecture suggests the following strategy for constructing a mirror.
\subsubsection{Strategy} \label{sec:strategy}
\begin{enumerate}[(i)]
\item produce a special Lagrangian fibration $\mu : X \to B$;\footnote{This choice is the reason that $X$ may have several mirrors.}
\item dualize the smooth locus $\mu_0 : X_0 \to B_0$ to obtain a semi-flat mirror $\check{\mu}_0: \check{X}_0 \to B_0$;
\item compactify $\check{X}_0$ to obtain a CY $n$-fold with a dual SYZ fibration $\check{\mu} : \check{X} \to B$;
\item use the geometry of the dual fibrations to construct a HMS equivalence $$\Phi: \mc{F}uk(X) \to D^b(\mathrm{Coh}(\check{X})).$$
\end{enumerate}

\subsubsection{Obstacles} There are many obstacles to carrying out \ref{sec:strategy} and (ii) is the only step where a satisfactory answer is known as we discussed in Section \ref{semiflat}. 

Producing sLag fibrations on a compact Calabi Yau $n$-folds is a hard open problem in general. Furthermore, work of Joyce \cite{joyce1} suggests that even when sLag fibrations exist, they might be ill-behaved. The map $\mu$ is not necessarily differentiable and may have real codimension one discriminant locus in the base $B$. In this case $B_0$ is disconnected and one needs to perform steps (ii) and (iii) on each component and then glue. 

Compactifying $\check{X}_0$ to a complex manifold also poses problems. There are obstructions to extending the semi-flat complex structure on $\check{X}_0$ to any compactification. To remedy this, one needs to take a small deformation of $\check{X}_0$ by modifying the complex structure using \emph{instanton corrections}. 

Step (iv) has been realized in some special cases (e.g. \cite{kontsevichsoibelman} \cite{ab1} \cite{ab2} \cite{ab3} \cite{blowup} and references therein) but a general theory for producing an equivalence $\Phi$ given an SYZ mirror is still elusive. 

\subsection{Instanton corrections}\label{sec:instanton} The small deformation of the complex structure on the dual $\check{X}_0$ is necessitated by the existence of obstructed Lagrangians. The point is that the Fukaya category of $X$ doesn't contain all pairs $(L,\nabla)$ of Lagrangians with flat connection but only those pairs where $L$ is \emph{unobstructed}. 

The differential for the Floer complex is constructed using a count of pseudoholomorphic discs bounded by $L$. In particular, $d^2$ is not necessarily zero. A Lagrangian $L$ is unobstructed if the necessary counts of pseudoholomorphic discs bounded by $L$ cancel out so that the Floer differential satisfies $d^2 = 0$. In particular, if $L$ doesn't bound any nonconstant holomorphic discs, then it is unobstructed. A problem arises if $\mu : X \to B$ has singular fibers because then the smooth torus fibers may bound nontrivial holomorphic discs known as \emph{disc instantons}. For example, any vanishing $1$-cycle on a nearby fiber sweeps out such a disc. 

To construct the dual $\check{X}$ as a complex moduli space of objects in the Fukaya we need to account for the effect of these instantons on the objects in the Fukaya category. This is done by modifying the semi-flat complex structure using counts of such disc instantons. 

In fact, one can explicitly write down the coordinates for the semi-flat complex structure described in Section \ref{semiflat} in terms of the symplectic area of cylinders swept out by isotopy of nearby smooth Lagrangian fibers as in Section \ref{LGmodel}. Then the discs bounded by obstructed Lagrangians lead to nontrivial monodromy of the semi-flat complex on $\check{X}_0$ which is an obstruction to the complex structure extending to a compactification $\check{X}$.  The instanton corrections are given by multiplying these coordinates by the generating series for virtual counts of holomorphic discs bounded by the fibers. 

For more details on instanton corrections, see for example \cite{auroux} \cite{blowup} \cite{tu}. 

\subsection{From torus fibrations to degenerations}  

Heuristics from physics suggest that $X$ will admit an SYZ fibration in the limit toward a maximally unipotent degeneration.\footnote{That is, a degeneration with maximally unipotent monodromy. These are sometimes known as large complex structure limits (LCSL).} It was independently conjectured in \cite{grosswilson} and \cite{kontsevichsoibelman} that if $\mc{X} \to \mb{D}$ is such a degeneration over a disc (where $X = \mc{X}_\epsilon$ for for some small $\epsilon \ll 1$) and $g_t$ is a suitably normalized metric on $\mc{X}_t$, then the Gromov-Hausdorff limit of the metric spaces $(\mc{X}_t, g_t)$ collapses the Lagrangian torus fibers onto the base $B$ of an SYZ fibration. Furthermore, this base should be recovered as the dual complex of the special fiber of $\mc{X} \to \mb{D}$ endowed with the appropriate singular integral affine structure. Then one can hope to reconstruct the instanton corrected SYZ dual directly from data on $B$. 

This allows one to bypass the issue of constructing a sLag fibration by instead constructing a maximally unipotent degeneration. Toric degenerations are particularly well suited for this purpose. This is the point of view taken in the Gross-Siebert program \cite{gs1} \cite{gs2} and gives rise to a version of SYZ mirror symmetry purely within algebraic geometry. In this setting the instanton corrections should come from logarithmic Gromov-Witten invariants of the degeneration as constructed in \cite{chen} \cite{ac} \cite{gs3} and these invariants can be computed tropically from data on the base $B$. For more on this see for example \cite{grosssurv} \cite{grosstrop} \cite{gs4}

\subsection{Beyond the Calabi-Yau case} \label{LGmodel}

The SYZ approach can also be used to understand mirror symmetry beyond the case of Calabi-Yau manifolds. The most natural generalization involves log Calabi-Yau pairs $(X,D)$ where $D \subset X$ is a boundary divisor and the sheaf $\omega_X(D)$ of top forms with logarithmic poles along $D$ is trivial. That is, $D$ is a section of the anticanonical sheaf $\omega_X^{-1}$ and $X \setminus D$ is an open Calabi-Yau. 

In this case the mirror should consist of a pair $(M,W)$ consisting of a complex manifold $M$ with a holomorphic function $W : M \to \CC$. The pair $(M,W)$ is known as a \emph{Landau-Ginzburg model} and the function $W$ is the \emph{superpotential} \cite{kapustinli}. Homological mirror symmetry takes the form of an equivalence 
$$
\Phi : \mc{F}uk(X,D) \to MF(M,W)
$$
between a version of the Fukaya category for pairs $(X,D)$ and the category of \emph{matrix factorizations} of $(M,W)$. Recall that a matrix factorization is a $2$-periodic complex
$$ 
\left(\xymatrix{\ldots \ar[r] & P_0 \ar[r]^d & P_1 \ar[r]^d & P_0 \ar[r] & \ldots}\right)
$$
of coherent sheaves on $M$ satisfying $d^2 = W$. By a theorem of Orlov \cite{orlov}, the category $MF(M,W)$ is equivalent to the derived category of singularities $D^b_{sing}(\{W = 0\})$.\footnote{Here we've assumed for simplicity that the only critical value of $W$ is at $0 \in \CC$.}

The SYZ conjecture gives a recipe for constructing the Landau-Ginzburg dual $(M,W)$. Here we give the version as stated in \cite{auroux1}:

\begin{conjecture} Let $(X, J, \omega)$ be a compact K\"ahler manifold and $D$ a section of $K_X^{-1}$. Suppose $\mu : U = X\setminus D \to B$ is an SYZ fibration where $U$ is equipped with a holomorphic volume form $\Omega$. Then the mirror to $(X,D)$ is the Landau-Ginzburg model $(\check{U}, W)$ where 
$$
\check{\mu} : \check{U} \to B
$$
is the SYZ dual fibration equipped with the instanton corrected complex structure and the superpotential $W$ is computed by counting holomorphic discs in $(X,D)$. 
\end{conjecture}

We briefly recall the construction of the superpotential. Let $\mu_0 : U_0 \to B_0$ be the smooth locus of the fibration so that $\check{U}_0$ is the semi-flat dual. Consider a family of relative homology classes $A_L \in H_2(X,L;\ZZ)$ as the Lagrangian torus fiber $L$ varies. Then the function
$$
z^A : \check{U}_0 \to \CC \enspace \enspace \enspace z^A(L,\nabla) = \exp\left(-\int_{A_L} \omega\right)\mathrm{hol}_\nabla(\partial A_L).
$$
is a holomorphic local coordinate on $\check{U}_0$. 

Let
$$
m_0(L,\nabla) = \sum_{\beta \in H_2(X,L;\ZZ)} n_\beta(L) z^\beta
$$
where $n_\beta(L)$ is Gromov-Witten count of holomorphic discs in $X$ bounded by $L$ and intersecting $D$ transversally. \footnote{More precisely, the sum is over curve classes $\beta$ with Maslov index $\mu(\beta) = 2$.} This is a holomorphic function on $\check{U}_0$ when it is defined but in general it only becomes well defined after instanton correcting the complex structure. The idea is that the number $n_\beta(L)$ jumps across an obstructed Lagrangian $L$ that bounds disc instantons in $X \setminus D$. Instanton corrections account for this and so $m_0$ should extend to a holomorphic function $W$ on the instanton corrected dual $\check{U}$. 
 
In fact $m_0$ is the obstruction to Floer homology constructed in \cite{fooo}. That is, $d^2 = m_0$ where $d$ is the Floer differential on the Floer complex $CF^*(L,L)$. This explains why the Landau-Ginzburg superpotential $W$ should be given by $m_0$. If one believes homological mirror symmetry, then obstructed chain complexes in the Fukaya category should lead to matrix factorizations with $W = m_0$ on the mirror. 

\begin{example} Let $X = \mb{P}^1$ with anticanonical divisor $\{0, \infty\} = D$. Then $U = \mb{C}^*$ admits a sLag fibration $\mu : U \to B$ where $B$ is the open interval $(0, \infty)$ and $\mu^{-1}(r) = \{|z| = r\}$ is a circle. The dual is $\check{U} = \mb{C}^*$ is also an algebraic torus and there are no instanton corrections since all the fibers of $\mu$ are smooth. Each sLag circle $L \subset U \subset X$ cuts $X$ into two discs $D_0$ and $D_\infty$ whose classes satisfy $[D_0] + [D_\infty] = [\mb{P}^1]$ in $H_2(X,L;\ZZ)$ so that the corresponding coordinate functions $z_0$ and $z_\infty$  on $\check{U}$ satisfy $z_0z_\infty = 1$. Furthermore, 
$$
\exp\left(-\int_{D_0} \omega\right) \exp\left(-\int_{D_\infty} \omega \right) = e^{-A}
$$
where $A = \int_{\mb{P}^1} \omega$ is the symplectic area. Furthermore, it is easy to see that $n_{[D_0]}(L) = n_{[D_1]}(L) = 1$. Putting it together and rescaling, we obtain the superpotential 
$$
W = z_0 + \frac{e^{-A}}{z_0} : \mb{C}^* \to \mb{C}. 
$$

A similar argument works for any Fano toric pair $(X,D)$ where $\mu$ is the moment map, $B$ is the interior of the moment polytope $P$, $\check{U} = (\mb{C}^*)^n$ is an algebraic torus, and $W$ is given as a sum over facets of $P$ \cite{auroux1} \cite{chooh}.
\end{example} 
\bibliographystyle{alpha} 
\bibliography{syz} 

\end{document}